\documentclass[reqno,11pt]{amsart}
\calclayout
\usepackage{amssymb, amsthm}
\usepackage[numbers]{natbib}
\usepackage{mathtools}
\usepackage[dvipsnames]{xcolor}
\usepackage{graphbox}
\usepackage[all]{xy}
\input epsf
\theoremstyle{plain}
\newtheorem{thm}{Theorem}[section]

\newtheorem{cor}[thm]{Corollary}
\newtheorem{prop}[thm]{Proposition}

\theoremstyle{definition}

\newtheorem{defn}[thm]{Definition}

\usepackage{setspace}

\DeclareMathOperator{\al}{\alpha}

\DeclareMathOperator{\om}{\omega}
\DeclareMathOperator{\sig}{\sigma}
\DeclareMathOperator{\Sig}{\Sigma}
\DeclareMathOperator{\im}{Im}
\DeclareMathOperator{\N}{\mathbb{N}}
\DeclareMathOperator{\Z}{\mathbb{Z}}
\DeclareMathOperator{\F}{\mathbb{F}}
\DeclareMathOperator{\rank}{rank}

\title{Winding number $m$ and $-m$ patterns acting on concordance}
\author{Allison N. Miller}
\begin{document}
\begin{abstract}
We prove that for any winding number $m>0$ pattern $P$ and winding number $-m$ pattern $Q$, there exist knots $K$ such that the minimal genus of a cobordism between $P(K)$ and $Q(K)$ is arbitrarily large. 
This answers a question posed by Cochran-Harvey [CH17]  and generalizes a result of Kim-Livingston [KL05].
\end{abstract}
\maketitle

\section{Introduction}

While most of the investigations of $\mathcal{C}$, the collection  of knots modulo concordance, have focused on its group structure, it is also natural to consider it as a metric space with metric $d(K,J):=g_4(K \#-J)$. 
Cochran and Harvey \cite{CH17} considered this geometric structure, focusing on the metric properties of maps induced by patterns in solid tori.  Following their work, we  consider the distance between two patterns, defined as 
\[ d(P,Q)= \sup_{K \in \mathcal{C}} d(P(K), Q(K)) \in \{0, 1, 2, \dots, \infty\}.
\]

It is natural to ask when two patterns are a finite distance from each other. Cochran and Harvey use Tristram-Levine signatures to give an almost complete characterization of this in terms of winding number. (For a discussion of pattern orientations, including a definition of winding number, we refer the reader to Section~\ref{section:background}.) All results stated here hold in both the smooth and the topological categories, since the constructions are smooth and the obstructions are topological. 
\begin{thm}[Cochran-Harvey \cite{CH17}]\label{thm:ch}
Let $P$ and $Q$ be patterns of winding number $m$ and $n$, respectively. If $n=m$ then $d(P,Q)$ is finite and if $|n| \neq |m|$ then $d(P,Q)$ is infinite. 
\end{thm}
We are therefore led to consider whether the distance between a winding number $m$ pattern and a winding number $-m$ pattern can ever be finite. Cochran-Harvey's arguments do not apply in this case: Tristram-Levine signatures are insensitive to the orientation of a knot, and for every winding number $m$ pattern $P$ there is a winding number $-m$ pattern $P^r$ such that $P(K)$ and $P^r(K)$ are always equal as unoriented knots. 
 Nevertheless, the case of $m=1$  was resolved by Kim and Livingston \cite{KimLivingston:2005} by using Casson-Gordon invariants, in a result that seems undeservedly forgotten. Note that the core of the torus, oriented one way, gives a winding number 1 satellite map $K \mapsto K$ and, oriented the other way, gives a winding number $-1$ satellite map $K \mapsto K^r$. 
\begin{thm}[Kim-Livingston \cite{KimLivingston:2005}]\label{thm:KL}
For any $g\geq 0$ there exists a knot $K$ such that $g_4(K\#-K^r)>g$. That is, the identity (winding number +1) and reversal (winding number -1) operators are infinite distance from each other. 
\end{thm}

It seems to have been assumed that the extension of Theorem~\ref{thm:KL} to the case of general $m~>~0$ would require substantial advances in the computation of Casson-Gordon invariants (see e.g. Remark 6.15 of \cite{CH17}). It is therefore perhaps somewhat surprising that we prove the following result while computationally only using Litherland's work of \cite{Lith84}; on the other hand, the potential relevance of formulae for Casson-Gordon invariants of satellite knots to the problem is clear. 

\begin{thm}\label{theorem:main}
Let $m>0$
 and  $P$ and $Q$ be patterns of winding number $m$ and $-m$, respectively. Then $d(P,Q)$ is infinite. 
\end{thm}

 Theorems~\ref{thm:ch} and~\ref{theorem:main}  combine to give the following.
 
\begin{cor}
Let $P$ and $Q$ be patterns of winding number $m$ and $n$, respectively. 
Then the distance between $P$ and $Q$ is finite if and only if $m=n$. 
\end{cor}

\section{Acknowledgements}
I would like to thank Chuck Livingston for helpful email correspondence and my advisor Cameron Gordon for his encouragement and thoughtful advice. 

\section{Background}\label{section:background}

Given an oriented knot $K$, a choice of $n \in \N$, and a map $\chi: H_1(\Sigma_n(K)) \to \Z_d$ on the first homology of the $n$th cyclic branched cover of $K$,  Casson-Gordon associate a rational number $\sig_{1} \tau(K, \chi)$, which is roughly the twisted signature of some associated 4-manifold \cite{CG86}.  We have the following key proposition relating the Casson-Gordon signatures of a knot to those of its mirror image (i.e. the concordance inverse of its reverse), which follows immediately from the basic definitions.
\begin{prop}\label{prop:reversal}
Let $K$ be a knot, $-K^r$ denote its mirror image, and $n \in \N$. 
Then there is a canonical isomorphism of groups $\al: H_1(\Sig_n(K)) \to H_1(\Sig_n(-K^r))$ such that
\begin{enumerate}
\item Letting $t_{K}$ and $t_{-K^r}$ denote the actions induced by the natural covering transformations on $H_1(\Sig_n(K))$ and $H_1(\Sig_n(-K^r))$, respectively,  we have  
$t_{-K^r} \cdot \al(x) = \al(t_K^{-1} \cdot x) $ for all $x \in H_1(\Sig_n(K))$.
\item Given $\chi: H_1(\Sig_n(K)) \to \Z_m$ we have $\sig_{1}\tau(-K^r, \chi \circ \al^{-1}) = - \sig_{1} \tau(K, \chi)$. 
\end{enumerate}
\end{prop}
Notice that if we replace $-K^r$ with $-K$, Part (1) of Proposition~\ref{prop:reversal} would be replaced with $t_{-K} \cdot \al(x) = \al(t_K \cdot x) $; since Casson-Gordon signatures are additive with respect to connected sums of knots, we would not be able to obtain any potential slice genus obstruction for $K \#-K$. This is reassuring, since $K \# -K$ is of course always slice. It also suggests to us that in order to obtain lower bounds for $g_4(K \# - K^r)$, we must pay particular attention to the action on the first homology induced by the covering transformation. We will use Gilmer's slice genus bound, in a slightly different form than originally stated. 
We use  $\sig_K(\om)$ to denote the Tristram-Levine signature of a knot $K$ at $\om \in S^1$ and for $n \in \mathbb{N}$ let $\omega_n:=e^{2\pi i/n}$. 

\begin{thm}[Gilmer \cite{Gil82}] \label{thm:Gilmer}
Let $K$ be a knot and suppose that $g_4(K) \leq g$.
Then for any prime power $n$ there is a decomposition of $H_1(\Sig_n(K))\cong A_1 \oplus A_2$ so that the following properties hold: 
\begin{enumerate}
\item  $A_1$ has a rank $2(n-1)g$ presentation with signature equal to $\sum_{i=1}^n \sig_K(\om_n^i)$. 
\item  $A_2$ has a subgroup $B$ such that $|B|^2=|A_2|$, and for any prime power order character $\chi:H_1(\Sig_n(K)) \to \Z_d$ which vanishes on $A_1 \oplus B$, we have   
\[| \sig_{1} \tau(K, \chi) + \sum_{i=1}^n \sig_K(\om_n^i) | \leq 2ng.
\]
Also, $A_1 \oplus B$ and $B$ are both covering transformation invariant. 
\end{enumerate}
\end{thm}
\begin{proof}
This follows from Gilmer's proof. Letting $W_n$ denote the $n$-fold cyclic branched cover of the 4-ball over the hypothesized genus $g$ surface with boundary $K$ and abbreviating $\Sig_n=\Sig_n(K)$,  we obtain $A_1$ and $B$ from the following exact sequence:
\[ 0 \to H_2(W_n) \to H_2(W_n, \Sig_n) \xrightarrow{\partial} H_1(\Sig_n) \to H_1(W_n) \to H_1(W_n, \Sig_n) \to 0.
\]
In particular, $A_1 \oplus B= \im(\partial)$ and $B= \im\left(\partial | _{TH_2(W_n, \Sig_n)}\right)$ are covering transformation invariant. 
\end{proof}
Note that Gilmer's original proof did not include any consideration of covering transformation invariance, due perhaps to the fact that his work explicitly dealt with the case $n=2$, when $t$  acts by multiplication by $-1$ and all subgroups are covering transformation invariant. Kim and Livingston's \cite{KimLivingston:2005}  proof that there exist knots $K$ for which $g_4(K \#-K^r)$ is arbitrarily large relies on this more general result in the case $n=3$. \\

Our examples are constructed via  various satellite operations. Given the importance of orientation in our context, we rather pedantically establish some orientation conventions pertaining to patterns in solid tori. 
Choose fixed orientations on $S^1$ and $D^2$. These induce  orientations on $V:= S^1 \times D^2$ and $\lambda_V:= S^1 \times \{x_0\}$, where $x_0$ is a (positively oriented) point in $\partial D^2$, as well as on $\mu_V:= \{y_0\} \times \partial D^2$.  These orientations for $V$,  $\lambda_V$, and $\mu_V$ will remain fixed throughout. Given a pattern $P:S^1 \to V$, the class of $P(S^1)$ in $H_1(V)$ is equal to $n[\lambda_V]$ for some $n \in \Z$. We call $n$ the \emph{(algebraic) winding number} of $P$.
To an oriented knot $K$ in $S^3$ we associate the positively oriented meridian $\mu_K$ and 0-framed longitude $\lambda_K$  in the standard way. Finally, note that as usual we mildly abuse notation by, for example,  referring to both the map $P$ and its image $P(S^1)$ as $P$. 

\begin{defn}\label{defn:pattern}
Given a knot $K$ in $S^3$ and a pattern $P:S^1 \to V$, define  the \emph{satellite knot} $P(K)$ as follows:
Let $i_K: V \to \overline{\nu(K)}\subset S^3$ be a homeomorphism with $i_K(\lambda_V)= \lambda_K$  and $i_K(\mu_V)= \mu_K$.
Then  $P(K):= i_K \circ P: S^1 \to S^3$. 
\end{defn}

\begin{figure}[h!]
\begin{minipage}{6.25in}
  \centering
  \includegraphics[align=c,height=3cm]{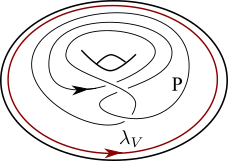}
  \hspace*{.2in}
  \includegraphics[align=c,height=3cm]{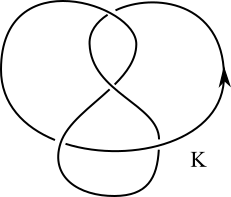}
  \hspace*{.2in}
  \includegraphics[align=c,height=4.5cm]{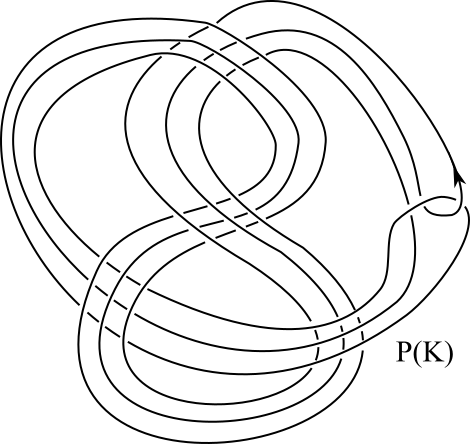}
\end{minipage}
\caption{A winding number +1 pattern $P$ in the solid torus $V$ with longitude $\lambda_V$ in red (left), a knot $K$ (center), and the satellite knot $P(K)$ (right).}
\label{orientedsatellite}
\end{figure}

Given a pattern $P:S^1 \to V$ of winding number $n$, we obtain a winding number $(-n)$ pattern $P^r$ by reversing the orientation of $S^1$ while fixing the orientations of $V$, $\lambda_V$, and $\mu_V$. Observe that $P^r(K)= (P(K))^r$, whereas $P(K^r)$ generally equals neither $P(K)$ nor $P^r(K)$. Our need for this plethora of orientations on $P$, $\lambda_V$, and $K$ in order to obtain a well-defined knot $P(K)$  is evident even in the simplest case: connected sum is not a well-defined operation on unoriented knots.

The work of Litherland~\cite{Lith84} completely describes the Casson-Gordon invariants of a satellite knot; we will only need the following special cases. 

\begin{thm}[Litherland \cite{Lith84}]\label{thm:Litherland0}
Let $P$ be a satellite operator, described via a curve $\gamma$ in the complement of $P(U)$ in $S^3$. 
Let $n \in \N$ be a prime power, and suppose that $\gamma$ has $n$ distinct lifts  $\gamma_1, \dots, \gamma_n$  to $\Sig_n(P(U))$.  
Then for any knot $K$ there is a canonical, covering transformation invariant isomorphism 
$\phi: H_1(\Sig_n(P(U))) \to  H_1(\Sig_n(P(K)))$ such that
for any prime power order character $\chi: H_1(\Sig_n(P(U))\to \mathbb{Z}_d$ we have 
\[
\sig_{1} \tau(P(K), \chi \circ \phi^{-1})= \sig_{1}\tau(P(U), \chi) + \sum_{i=1}^n \sig_K\left(\om_d^{\chi(\gamma_i)}\right).
\]
\end{thm}

\begin{thm}[Litherland \cite{Lith84}]\label{thm: litherlandrelprime} 
Let $P$ be a winding number $m$ satellite operator with $P(U)=U$ and  suppose $n \in \N$  is a prime power such that $(m,n)=1$. 
Then for any knot $K$ there is a canonical, covering transformation invariant isomorphism $\phi:  H_1(\Sig_n(K)) \to H_1(\Sig_n(P(K)))$
such that for any prime power order character $\chi: H_1(\Sig_n(K)) \to \mathbb{Z}_d$  we have
\[\sig_1\tau(P(K), \chi \circ \phi^{-1} )= \sig_1 \tau(K, \chi).\]
\end{thm}

\section{Winding number $m$ and $-m$ patterns are  unbounded distance in their action on concordance.}

Let $C_{m,1}$ denote the $(m,1)$ cabling pattern and $C_{m,1}^r$ denote the winding number $-m$ pattern obtained by reversing $C_{m,1}$. 

\begin{prop}\label{cor:cabling}
Suppose $K$ is a knot such that $n$-fold branched cover Casson-Gordon signature obstructions show that 
$g_4(K \# -K^r)>g$. 
Then for any $m$ which is relatively prime to $n$ we have that
$g_4(C_{m,1}(K) \# -C_{m,1}^r(K))>g$ too. 
\end{prop}
\begin{proof}
First, observe that by Theorem~\ref{thm: litherlandrelprime} we have a canonical, covering transformation invariant correspondence between the Casson-Gordon signatures of $K$ corresponding to the $n$-fold branched cover and those of $C_{m,1}(K)$. 
So the $n$-fold branched cover Casson-Gordon signature obstructions show that 
$g_4(C_{m,1}(K) \# -(C_{m,1}(K))^r)>g$. 
But $-(C_{m,1}(K))^r= -(C_{m,1}^r(K)). $
\end{proof}

We will show in Theorem~\ref{thm:exampleswork} that
for any odd prime $p$ and any $g \in \N$, there exists a knot $K$ such that $g_4(K \#-K^r) >g$ as detected by $p$-fold cyclic branched cover Casson-Gordon signatures. Once we have this result, Theorem~\ref{theorem:main} follows. 

\begin{proof}[Proof of Theorem~\ref{theorem:main}]
Fix $m>0$. Let $P$  and $Q$ be arbitrary patterns of winding number $m$ and $-m$, respectively.
Observe that Theorem~\ref{thm:ch} implies that $d(C_{m,1}, P)$ and $d(C_{m,1}^r, Q)$ are both finite, so it suffices to show that $d(C_{m,1}, C_{m,1}^r)$ is infinite. 
Let $g\geq 0$ be given, and let $p$ be an odd prime which does not divide $m$. By Theorem~\ref{thm:exampleswork}, there exists a knot $K$ such that the $p$th cyclic branched cover Casson-Gordon signatures show that $g_4(K \#-K^r)>g$. By Proposition~\ref{cor:cabling}, we therefore have that $g_4(C_{m,1}(K) \# - C_{m,1}^r(K))>g$, too. 
\end{proof}

For a fixed $p$ and $g$, we will take $K= \#^{g+1} J_g$, where  $J_g$ is obtained by iterated satellite operations along a $(p-1)$-component unlink $\{\eta_j\}_{j=1}^{p-1}$ in the complement of some knot $J_0$. The key property of $J_0$ will be that  for some prime $q$ and distinct $a_1, \dots, a_{p-1} \in \F_q$,
\[H_1(\Sig_p(J_0), \F_q) \cong \F_q[t]/ \langle \Phi_p(t) \rangle \cong \bigoplus_{j=1}^{p-1} \F_q[t]/ \langle t- a_j\rangle.\]
(Here $\Phi_p(t)=t^{p-1}+t^{p-2} + \dots + t +1$.) 
Each curve $\eta_j$  will correspond to  a generator of the $\F_q[t]/ \langle t- a_j \rangle$- summand above, in a way we will make precise. 

\begin{prop}\label{prop:knotexists}
For any odd prime $p$, there exists a prime $q$ and a knot $J$ such that 
\[H_1(\Sig_p(J), \F_q) \cong \F_q[t]/ \langle \Phi_p(t) \rangle \cong \bigoplus_{j=1}^{p-1} \F_q[t] / \langle t- a_j \rangle,\]
where $a_1, \dots, a_{p-1}$ are distinct elements of $\F_q$ and $a_{p-j} \equiv a_j^{-1} \mod q$. 
\end{prop}
\begin{proof}
Let $q$ be a prime which is equivalent to $1$ mod $p$, so we can write $q=kp+1$ for some $k \in \N$. Note that $0<k<q$ and so $k$ is a unit mod $q$. 
Let $a(t)= k\Phi_p(t) -q t^{\frac{p-1}{2}}$. Observe that $a(t)$ is a symmetric polynomial with $a(1)= kp-q=-1$. Levine's work \cite{Levine:1965}  characterizing the Alexander polynomials of knots implies that there is a knot with Alexander polynomial equal to $a(t)$. In fact, his construction gives 
a knot $J$ with Alexander module given by $H_1(\widetilde{X_J}) \cong \Z[t, t^{-1}]/ \langle a(t) \rangle$. So
\[ H_1(\Sig_p(J), \F_q) \cong H_1(\widetilde{X_J}, \F_q)/ \langle t^p-1 \rangle  \cong \F_q[t]/ \langle a(t), t^p-1 \rangle \cong \F_q[t]/ \langle k \Phi_p(t), t^p-1 \rangle \cong \F_q[t]/ \langle \Phi_p(t) \rangle.
\]
It is a standard fact of number theory that since the order of $q$ mod $p$ is 1, $\Phi_p(t)$ splits into linear factors over $\F_q$. In addition, $\Phi_p(t)$ has  distinct roots, as one can easily verify by considering $f(t)=(t-1) \Phi_p(t)= t^p-1$. Since the only root of $f'(t)= pt^{p-1}$ over $\F_q$ is $t=0$, we have that $f'(t)$ and $f(t)$ have no common roots and so $f(t)$ has no repeated roots over $\F_q$. So $\Phi_p(t)$ certainly cannot have repeated roots either, and there are distinct $a_1, \dots, a_{p-1
} \in \F_q$ such that
\[H_1(\Sig_p(J), \F_q) \cong \F_q[t]/ \langle \Phi_p(t) \rangle\cong \bigoplus_{j=1}^{p-1} \F_q[t]/\langle t- a_j \rangle.\]
Note that this decomposition is canonical, since the $\F_q[t]/ \langle t-a_j \rangle$ summand is exactly the eigenspace of the action of $t$ corresponding to eigenvalue $a_j$. Since $\Phi_p(a)=0$ if and only if $\Phi_p(a^{-1})=0$, after reordering we can also assume that $a_{p-j} \equiv a_j^{-1} \mod q$. 
\end{proof} 

Now fix an odd prime $p$ and let $J_0$ be as in Proposition~\ref{prop:knotexists}. 
For $j=1, \dots, p-1$, let $x_j \in H_1(\Sig_p(J_0), \F_q)$ be an arbitrary generator of the $\F_q[t]/ \langle t - a_j \rangle$ summand (i.e., $x_j$ is an eigenvector of the covering transformation induced action on $H_1(\Sig_p(J_0), \F_q)$ with eigenvalue $a_j$).   Choose elements $\alpha_j \in \pi_1(X_p(J_0)) \subseteq \pi_1(X(J_0))$ which map to $x_j$ under the natural map $\pi_1(X_p(J_0)) \to \pi_1(\Sig_p(J_0)) \to H_1(\Sig_p(J_0)) \to H_1(\Sig_p(J_0)), \F_q)$. 
 Now choose curves $\eta_1, \dots, \eta_{p-1}$ in the complement of $J_0$
 such that $\eta_j$ represents $a_j$  in $\pi_1(X(J_0))$ for each $j=1, \dots, p-1$.
Notice that changing the crossings of the $\eta_j$ curves with each other does not change this property, and so by crossing changes we can assume that $\cup_{j=1}^{p-1} \eta_i$ is an unlink in $S^3$. For a choice of knots $A_1, \dots, A_{p-1}$, denote by $J(A_1, \dots, A_{p-1})$ the knot obtained by infecting $J_0$ by $A_i$ along $\eta_j$ for $j=1, \dots p-1$. Note that since $\cup_{j=1}^{p-1} \eta_i$ is an unlink we can consider this infection as a $(p-1)$-fold iterated satellite operation, and Theorem~\ref{thm:Litherland0} applies. In particular, observe that for each $j$ the homology classes of the $p$ lifts of $\eta_j$ are given by $\{t^k x_j= a_j^k x_j \}_{k=1}^p$. Theorem~\ref{thm:Litherland0} then implies that given any character $\chi: H_1(\Sig_p(J_0)) \to \Z_q$, and under the natural identification of $H_1(\Sig_p(J(A_1, \dots A_{p-1})))$ with $H_1(\Sig_p(J_0))$, we have 
\begin{align}\label{eq:cgformulaj}
\sig_1 \tau(J(A_1, \dots A_{p-1}), \chi) &= \sig_1 \tau(J_0, \chi) + \sum_{j=1}^{p-1} \left[ \sum_{k=1}^p \sig_{A_j}\left(\om_q^{\chi(t^k x_j)}\right) \right]
\end{align}
%
%

By the proof of Theorem 1 of Cha-Livingston \cite{Cha-Livingston:2004}, for any $\omega \in S^1$ there is some $\omega' \in S^1$ arbitrarily close to $\omega$ and a knot $K$ whose jumps in the Tristram-Levine signature function occur exactly at $\omega'$ and $\overline{\omega'}$. In particular, there exists a knot $C$ such that the only jumps in the Tristram-Levine signature $\sig_C(\omega)$ occur just before $\omega_q$ and just after $\overline{\omega_q}$, and hence one such that 
\begin{align*} 
\sig_C(\omega_q^k)= 
\left\{
 \begin{array}{ll} 
0 & k\equiv 0 \mod q \\
 \sig_C(\omega_q)> 0 & k \not \equiv 0 \mod q
  \end{array} 
    \right. .
\end{align*}

Now also fix $g \geq 0$. Let $\displaystyle c= \max_{\chi: H_1(\Sig_p(J_0)) \to \Z_q}\{ |\sig_1 \tau(J_0, \chi)|\}$. By taking sufficiently large connected sums of $C$, we can obtain knots $A$ and $B$ such that for all $1 \leq i \leq q-1$ we have
\begin{align*}
p\sig_{A}(\om_q^i)&= p\sig_A(\om_q)> 2(g+1)c+2pg \\
p\sig_B(\om_q^i)&=p\sig_B(\om_q)>(g+1)p \sig_A(\om_q)+ 2(g+1)c+2pg.
\end{align*}
Recall that for any $1 \leq j \leq p-1$ and $1 \leq k \leq p$,  we have $t^k x_j= a_j^k x_j$ for some nonzero eigenvalue $a_j$. It follow that  $\chi(t^k x_j)=\chi(a_j^k x_j)= a_j^k \chi(x_j)$ is congruent to 0 mod $q$ if and only if $\chi(x_j) \equiv 0 \mod q$. 
 It follows that for $1 \leq j \leq p-1$ we have 
$ \displaystyle
\sum_{k=1}^{p} \sig_{A}\left(\om_q^{\chi(t^k x_j)}\right) = \left\{ \begin{array}{cc}
p \sig_A(\om_q) & \text{ if } \chi(x_j) \not \equiv 0 \mod q \\
0 & \text{ if } \chi(x_j) \equiv 0 \mod q
\end{array}
\right. $, as well as an analogous formula for $B$. 

We choose $A_1= A_2 = \dots = A_{(p-1)/2}=A$ and $A_{(p+1)/2}= \dots= A_{p-1}=B$, and let $J_{A,B}= J(A_1, \dots, A_{p-1})$. The key point here is that since $a_{j}^{-1} \equiv a_{p-j} \mod q$ for all $j$, we  infect curves corresponding to eigenvalues $a$ and $a^{-1}$ with different knots.

 Now define  $\displaystyle \delta_j(\chi)= \left\{ \begin{array}{cc} 1 &\text{ if }\chi(x_j) \not \equiv 0 \mod q \\ 0 &\text{ if }\ \chi(x_j) \equiv 0 \mod q \end{array} \right.$, 
and observe that Equation~\ref{eq:cgformulaj} becomes
\begin{align}\label{eq:cgformula2}
\sig_1 \tau(J_{A,B}, \chi)= \sig_1 \tau(J_0, \chi) +
p \sig_A(\om_q) \sum_{j=1}^{\frac{p-1}{2}} \delta_j(\chi)+
p \sig_B(\om_q) \sum_{j=\frac{p+1}{2}}^{p-1} \delta_j(\chi)
\end{align}


\begin{thm}\label{thm:exampleswork}
For fixed odd $p$ and $g\geq 0$, let $J_g= J_{A,B}$ be as above, and let $K= \#^{g+1} J_g$. 
Then the Casson-Gordon signatures associated to the $p$th cyclic branched cover show that $g_4(K \# -K^r)>g$. 
\end{thm}

\begin{proof}
As in Proposition~\ref{prop:reversal}, our identification $H_1(\Sig_p(J_g), \F_q) \cong H_1(\Sig_p(J_0), \F_q) \cong \bigoplus_{j=1}^{p-1}\left( \F_q[t]/ \langle t- a_j \rangle \right) \langle x_j\rangle$ induces a description of $H_1(\Sig_p(-J_g^r), \F_q)$ as $\bigoplus_{j=1}^{p-1}\left( \F_q[t]/ \langle t^{-1}- a_j \rangle \right) \langle y_j\rangle$, such that 
$ \sig_1 \tau(-J_g^r, \chi: y_j \mapsto c_j) = - \sig_1\tau(J_g, \chi: x_j \mapsto c_j).$
We therefore have that
\begin{align}\label{H1}
H_1(\Sig_p(K \#-K^r), \F_q)&\cong \bigoplus_{i=1}^{g+1} H_1(\Sig_p(J_g), \F_q) \oplus \bigoplus_{i=1}^{g+1} H_1(\Sig_p(-J_g^r), \F_q) \\
&= \bigoplus_{i=1}^{g+1} \bigoplus_{j=1}^{p-1}\left( \F_q[t]/ \langle t- a_j \rangle \right) \langle x^i_j\rangle \oplus \bigoplus_{i=1}^{g+1} \bigoplus_{j=1}^{p-1}\left( \F_q[t]/ \langle t^{-1}- a_j \rangle \right) \langle y^i_j\rangle
\end{align}

For any $\chi= \bigoplus_{i=1}^{g+1} \chi_i \oplus \bigoplus_{i=1}^{g+1}\chi_i'$ and for $1 \leq j \leq p-1$, define $n_j(\chi)$ and $n_j'(\chi)$ as follows: 
$ n_j(\chi)= \sum_{i=1}^{g+1} \delta_j(\chi_i)$ and 
$n_j'(\chi)=  \sum_{i=1}^{g+1} \delta_j(\chi_i')$.
By the additivity of Casson-Gordon signatures and 
Equation~\ref{eq:cgformula2}, we have that 
\begin{align*}
 \label{nexteq}
 \sig_1 \tau (K \# -K^r, \chi) &= \sum_{i=1}^{g+1} \sig_1 \tau(J, \chi_i) + \sum_{i=1}^{g+1} \sig_1\tau(-J^r, \chi_i') \\
&=  \sum_{i=1}^{g+1} \left(\sig_1 \tau(J_0, \chi_i) + \sum_{j=1}^{p-1} \delta_j(\chi_i) p \sig_{A_j}(\om_q) \right) -
 \sum_{i=1}^{g+1} \left(\sig_1 \tau(J_0, \chi_i') + \sum_{j=1}^{p-1} \delta_j(\chi_i') p \sig_{A_j}(\om_q) \right)
\\
&= \sum_{i=1}^{g+1} \left( \sig_1 \tau(J_0, \chi_i)- \sig_1 \tau(J_0, \chi_i')\right) 
 + \sum_{j=1}^{p-1} \left(n_j(\chi)- n_j'(\chi)\right) p \sig_{A_j}(\om_q)\\
 &=\sum_{i=1}^{g+1} \left( \sig_1 \tau(J_0, \chi_i)- \sig_1 \tau(J_0, \chi_i')\right) 
 + p \sig_A(\om_q) \sum_{j=1}^{\frac{p-1}{2}} \left(n_j(\chi)- n_j'(\chi)\right)\\
& \hspace{6cm} + p \sig_B(\om_q) \sum_{j=\frac{p+1}{2}}^{p-1} \left(n_j(\chi)- n_j'(\chi)\right)
\end{align*}

Note that  $H_1(\Sig_p(K \#-K^r), \F_q)$ is isomorphic as a group to $\F_q^{(p-1)(2g+2)}$, and so a subgroup $H$ has rank  $r$ if and only if it has order  $q^r$. We wish to apply Theorem~\ref{thm:Gilmer} to conclude that $g_4(K \#-K^r)>g$, and  it is easy to check that it suffices to prove the following claim. 
\\

\noindent {\bf Claim:}
For every covering transformation invariant subgroup $H \leq H_1(\Sig_p(K \#-K^r), \F_q)$ of rank $(p-1)(2g+1)$ 
there exists $\chi:  H_1(\Sigma_p(K \# -K^r)) \to \F_q$ which vanishes on $H$ such that 
$|\sig_1\tau(K \#-K^r, \chi)| >2pg.$
\\

Let $H$ be as in the claim. Since $H$ is an invariant subspace and $H_1(\Sig_p(K \#-K^r), \F_q)$ is spanned by eigenvectors, $H$ has a basis of eigenvectors $\beta'$, as proven for instance in \cite{KL99b}. 
Let $B_j$ be the $a_j$-eigenspace of the covering transformation induced action on  $H_1(\Sig_p(K \#-K^r), \F_q)$, for $1 \leq j \leq p-1$.  Note $B_j$ has a basis $\beta_j= \{x^i_j\}_{i=1}^{g+1} \sqcup \{y^i_{p-j}\}_{i=1}^{g+1}$, and is rank $2g+2$. Since $H$ is spanned by eigenvectors, we have that 
\[ \sum_{j=1}^{p-1} \rank( B_j \cap H) = \rank(H)= (p-1)(2g+1).\]

Since $H \neq H_1(\Sig_p(K \#-K^r), \F_q)$ there is some $j_0$  such that $B_{j_0} \not \subset H$. Assume without loss of generality that $j_0 \leq \frac{p-1}{2}$. 
 Let $v_1$ be in $B_{j_0}$ but not in $H$. We can extend $\beta' \sqcup \{v_1\}$ to a basis $\beta''$ of eigenvectors for $H_1(\Sig_p(K \#-K^r), \F_q)$ by adding some $p-2$ vectors, $v_2, \dots, v_{p-1}$. 
Let $\chi$ be defined as follows on elements of $\beta''$, and extended linearly over $H_1(\Sig_p(K \#-K^r), \F_q)$:
\begin{align*}
\chi(v)= \left\{ \begin{array}{ll}
0 & \text{if } v \in \beta' \\
1 & \text{if } v=v_1 \\
0 & \text{if } v=v_i \text{ for } i =2, \dots, p-1
\end{array}\right..
\end{align*}
Observe that $\chi$ vanishes both on $H$ and on $B_j$ for all $j \neq j_0$.
We therefore have that $n_j(\chi)=0$ for $j \neq j_0$ and $n_j'(\chi) =0$ for $j \neq p-j_0$. Our formula for $\sig_1 \tau (K \# -K^r, \chi)$ therefore becomes
 \begin{align*}
 \sig_1 \tau (K \# -K^r, \chi)
 =\sum_{i=1}^{g+1} \left( \sig_1 \tau(J_0, \chi_i)- \sig_1 \tau(J_0, \chi_i')\right) 
 &+ p \sig_A(\om_q) n_{j_0}(\chi) -  p \sig_B(\om_q)n_{p-j_0}'(\chi).
\end{align*}
Since $\chi$ is not the zero character we must have that one of $n_{j_0}(\chi)$ and $n_{p-j_0}'(\chi)$ is  positive. 

Case 1: $n_{p-j_0}'(\chi)> 0$. 
Then, noting that $n_{j_0}(\chi) \leq g+1$, by our choice of $\sig_B(\om_q)$  we have 
 \begin{align*}
 \sig_1 \tau (K \# -K^r, \chi)
& \leq 2(g+1)c + p \sig_A(\om_q) (g+1) -  p \sig_B(\om_q)< -2pg.
\end{align*}

Case 2: $n_{p-j_0}'(\chi)=0$ and $n_{j_0}(\chi) >0$. 
Then by our choice of $\sig_A(\om_q)$ we have 
\begin{align*}
\sig_1 \tau(K \#-K^r, \chi) \geq -2(g+1)c+p \sig_A(\om_q) > 2pg. 
\qquad \qquad \qedhere
\end{align*}
\end{proof}

\bibliographystyle{alpha}
\bibliography{thesisbib}

\end{document}